\documentclass[10pt,reqno]{amsart}
\usepackage{amsmath}
\usepackage{amsfonts}
\usepackage{mathrsfs}
\usepackage{amssymb}
\usepackage{xypic}
\usepackage{amsthm}
\usepackage{url}
\usepackage{setspace}     \spacing{1}
\usepackage{enumitem}
\usepackage{times}
\usepackage{bbm}
\usepackage[colorlinks=true]{hyperref} \hypersetup{urlcolor=blue, citecolor=blue,linkcolor=blue}

\theoremstyle{Theorem}
\newtheorem{theorem}{Theorem} [section]

\newtheorem{proposition}[theorem]{Proposition} 
\newtheorem{claim}[theorem]{Claim} 
\newtheorem{lemma}[theorem]{Lemma}

\newtheorem{corollary}[theorem]{Corollary}
 
\theoremstyle{definition}

\newtheorem{remark}[theorem]{Remark}

\theoremstyle{remark}
\usepackage{mdframed}
\usepackage{pstricks, pst-plot, pst-node, pst-math,pst-3dplot,pstricks-add}
\newlist{enumlemma}{enumerate}{3} 
\setlist[enumlemma]{label*=\emph{(\roman*)}, ref= {\emph{(\roman*)} }}
\newlist{enum2}{enumerate}{3} 
\setlist[enum2]{label*={(\alph*)}, ref= {(\alph*)} }

\usepackage{marginnote}

\def\DPP{d}

\newcommand{\hol}[1][\beta]{ \, \mathrm{H\ddot{o}l}^{{#1}}}
\newcommand{\lochol}[1][\beta]{ \, \mathrm{H\ddot{o}l}^{{#1}}_{\mathrm{loc}}}
\def\Hol{\hol}

\newcommand{\Lip}{\mathrm{Lip}}
\def\lip{\Lip}

\renewcommand{\bar}{\overline}

\newcommand{\R}{\mathbb {R}}

\newcommand{\Z}{\mathbb {Z}}
\newcommand{\N}{\mathbb {N}}

\newcommand{\inv}{^{-1}}
\newcommand{\td}{\tilde}

\title[Smoothness of  stable holonomies]{Smoothness of  stable holonomies inside  center-stable manifolds and the $C^2$ hypothesis in Pugh--Shub and Ledrappier--Young theory} 
\date{\today}
\author[A.~Brown]{Aaron Brown}
\address{University of Chicago, Chicago, IL 60637, USA}
\email{awb@uchicago.edu}

\long\def\symbolfootnote[#1]#2{\begingroup\def\thefootnote{\fnsymbol{footnote}}
\footnote[#1]{#2}\endgroup}
\begin{document}
\maketitle
\begin{abstract}
Under  a suitable bunching condition, we establish  that stable holonomies inside center-stable manifolds   for $C^{1+\beta}$ diffeomorphisms are uniformly bi-Lipschitz and in fact $C^{1+\text{H\"older}}$.  This verifies that the Pugh--Shub theory for ergodicity holds for suitably center-bunched, $C^{1+\beta}$, essentially accessible, partially hyperbolic diffeomorphism and verifies that the Ledrappier--Young  entropy formulas hold for $C^{1+\beta}$ diffeomorphisms of compact manifolds.  
\end{abstract}
\section{Introduction}
The purpose of this   paper is twofold.  First, in \cite{MR2630044}, Burns and Wilkinson establish the ergodicity (and $K$-property) of  partially hyperbolic, center-bunched, essentially accessible $C^2$-diffeomorphisms.  This extends a number of earlier results including \cite{MR1298715} and \cite{MR1750453}.   A similar result (with stronger center-bunching conditions) was announced for $C^{1+\delta}$-diffeomorphisms.  However, it seems that the bunching condition  given in \cite[Theorem 0.3]{MR2630044} is possibly too weak for the method of proof.  
A proof of the technical result needed to establish \cite[Theorem 0.3]{MR2630044}
was circulated as an unpublished note in \cite{Burns_anote}.  It seems some of the details of the proof in \cite{Burns_anote} are incorrect.   
We correct the details of the proof of the main result in  \cite{Burns_anote} and obtain a proof of \cite[Theorem 0.3]{MR2630044} under somewhat stronger bunching hypotheses.  

Secondly, in two seminal papers \cite{MR819556, MR819557} Ledrappier and Young prove remarkable results relating the metric entropy of a $C^2$ diffeomorphism $f\colon M\to M$  of a compact manifold $M,$ its Lyapunov exponents, and the geometry of conditional measures along unstable manifolds.  In \cite{MR819556},  the SRB property of   measures satisfying the Pesin entropy formula  is established   for $C^2$ diffeomorphisms and   measures  with zero Lyapunov exponents.  This extends Ledrappier's result from \cite{MR743818} which established the SRB property for hyperbolic measures  invariant under $C^{1+\beta}$ diffeomorphisms satisfying the Pesin entropy formula. In \cite{MR819557}, a more general formula for the entropy of arbitrary measures invariant under a $C^2$ diffeomorphism is derived.  

 As remarked in  \cite{MR819556}, there is one crucial step in which the $C^2$ hypothesis rather than the $C^{1+\beta}$ hypothesis on the dynamics is used: establishing the Lipschitzness of unstable holonomies inside center-unstable sets. 
 In \cite{MR819557}, the corresponding estimate is the Lipschitzness of the holonomies along intermediate unstable foliations inside the total unstable manifolds.  
In the case of hyperbolic measures, the entropy formula from \cite{MR819557} is known to hold for $C^{1+\beta}$ diffeomorphisms as it is sufficient to establish the Lipschitzness of  $W^i$ holonomies inside the $W^{i+1}$ manifolds (corresponding to Lyapunov exponents $\lambda_i>\lambda_{i+1} >0$) on Pesin sets.  This Lipschitzness of holonomies along  intermediate unstable manifolds was established in \cite[Appendix]{MR1709302}.   
However, the proof in \cite[Appendix]{MR1709302} fails to establish Lipschitzness of unstable holonomies     inside center-unstable sets which is needed to show  the main result of \cite{MR819556}:  that the entropy of $f$ is ``carried entirely by the unstable manifolds.''
The results of this note establish the  Lipschitzness of unstable holonomies     inside center-unstable sets 
which confirms that the results of  \cite{MR819556, MR819557} hold  for $C^{1+\beta}$ diffeomorphisms and invariant measures with zero Lyapunov exponents.  

In Section \ref{Sec:2} we   establish under certain bunching conditions that the stable holonomies for certain globalized $C^{1+\beta}$  dynamics  are Lipschitz and, in fact, $C^{1+\text{H\"older}}$.    
The main result, Theorem \ref{thm:difflholonomies}, is formulated  in a sufficiently abstract setting  so that it may be applied to a number of settings.   In Sections \ref{sec:3} and  \ref{sec:4} we briefly justify how the results discussed above   follow from Theorem \ref{thm:difflholonomies}.

\section{Statement of main theorem}
\label{sec:1}
\subsection{Setup}
For every $n\in \Z$, let $$L_n = \left(\begin{array}{ccc}A_n & 0 & 0 \\0 & B_n & 0 \\0 & 0 & C_n\end{array}\right)$$ be an invertible linear map  where  each   $A_n$, $B_n$, and $C_n$ is a square matrix with dimension constant in $n$.  We assume the matrices are non-trivial  though the results can be formulated (with fewer conditions) in the case that  $C_n$ vanishes.
  We assume there are constants $$ -\mu<  \eta'_n<  \kappa'_n<   \gamma'_n\le \hat \gamma'_n<  \hat \kappa' _n<\hat \eta'_n <\mu $$ such that for every $n$
\begin{enumerate}
\item $e^{\eta '_n}\le m(A_n) \le \|A_n\| \le e^{\kappa '_n}$;
\item $e^{ \gamma'_n}\le m(B_n) \le \|B_n\| \le e^{ \hat\gamma'_n}$;  

\item $e^{\hat \kappa'_n}\le m(C_n) \le \|C_n\| \le e^{ \hat \eta'_n}.$
\end{enumerate}
Here $\|\cdot \|$ is the operator norm induced by the standard   norm on the corresponding Euclidean spaces  and $m(A)$ denotes the conorm of $A$.  
Throughout, we will always assume that $\sup\{ \kappa'_n\}<0$. 
We do not impose any assumptions on the signs of $\gamma'_n$, $\hat \gamma'_n, \hat \kappa'_n$ and $\hat \eta'_n$.  
We moreover assume  that   $$\inf \{ |\kappa'_n-\gamma'_n|, \hat \kappa'_n-\hat \gamma'_n ,     | \kappa_n'|  \}>0.$$
Take any $$0<\epsilon'\le \inf \{|\kappa'_n-\gamma'_n|,\hat \kappa'_n-\hat \gamma'_n , |\kappa'_n|, 1 \}/100$$ 
and let $\kappa_n= \kappa'_n+ 2\epsilon'$, $\hat \kappa_n=\hat  \kappa'_n- 2\epsilon'$, $\eta_n= \eta'_n- 2 \epsilon'$, $\hat\eta_n= \hat \eta'_n+ 2 \epsilon'$,  $ \gamma_n= \gamma'_n- 2\epsilon'$, $\hat  \gamma_n= \hat \gamma'_n+ 2\epsilon'$, and $\epsilon = 4\epsilon'.$
Given a fixed $0<\beta<1$   we moreover assume that $ \hat\gamma' _n-\gamma '_n$ and $\epsilon'$ are sufficiently small so that there exist   $0<\bar \theta<\beta $ satisfying   
\begin{equation}\label{eq:holder} 
\sup\dfrac{e^{\kappa_n}}{e^{\eta_n \bar \theta +\gamma_n}} <1 \text{ and } \sup\dfrac{e^{-\hat \kappa_n}}{e^{-\hat \eta_n \bar \theta -\hat \gamma_n}} <1
\end{equation}
 and  $\theta\le \bar \theta$  with 
\begin{equation}\label{eq:coarsbunch}
\sup\dfrac{e^{\beta \kappa_n}}{e^{\kappa_n \theta +\beta \gamma_n}}  <1\end{equation}
and 
\begin{equation}\label{eq:bunch}
\sup \frac{ \hat \gamma _n-   \gamma_n}{-\kappa_n }< \theta. \end{equation}
Condition \eqref{eq:bunch} is a bunching condition.  Condition \eqref{eq:holder}  ensures the tangent distributions defined below are $\bar \theta$-H\"older.
Note that with $\theta =\bar \theta$, \eqref{eq:bunch}  is the bunching condition stated in \cite[Theorem 0.3]{MR2630044}.  Our proof however requires the extra bunching   imposed by  \eqref{eq:coarsbunch}.  Note from \eqref{eq:holder} that $\theta = \beta\bar \theta$ satisfies \eqref{eq:coarsbunch}.

We fix such $0<\theta\le \bar \theta<\beta$ for the remainder. 
Also fix $\alpha>0$ and $\hat \theta< \theta \le \bar \theta< \hat \beta<\beta$ for the remainder with     \begin{equation}\label{eq:3} 
\sup \left\{ \frac{ (1+\alpha)(\hat\gamma _n-  \gamma_n)}{-\kappa_n }\right\}<\hat  \theta, \text{ and } 
\sup\left\{\dfrac{  \kappa_n}{\kappa_n- \gamma_n}\theta\right\}  <\hat \beta.\end{equation} 
Set $\bar \kappa = \sup  \{\kappa_n\} <0 $ and $$  \omega = \sup\{\kappa_n   \theta + (1+\alpha)(\hat\gamma _n-  \gamma_n)\}<0, \quad 
 \hat \omega = \sup\{\kappa_n\hat  \theta + (1+\alpha)(\hat\gamma _n-  \gamma_n)\}<0 .$$


\def\PRN{\mathbb{PR}^N}

We decompose $\R^k$ into subvector spaces $\R^k= \R^s\oplus  \R^c\oplus  \R^u$ according to the block decomposition preserved by each $L_n$.
We let $\|\cdot\|$ denote the standard Euclidean norm on $\R^k$ 
and write $d$ for the induced  distance. 
 Given a subspace $U\subset \R^k$ we write $S U$ for the unit sphere in $U$ relative to the Euclidean norm $\|\cdot\|$.   
 If $T\colon U\to V$ is linear we write $ T_*\colon S U\to S V$ for the induced map.  We recall that if $T\colon U\to V$ is a  linear isomorphism with  $a\le m(T)\le \|T\|\le  b$ then $ T_* $ is bi-Lipschitz 
 with constants $b\inv a$ and  $ba\inv$.  
Finally, if $N\subset \R^k$ is an embedded submanifold we write $SN:= STN$ for the sphere bundle over $N$.  
Given $g\colon N_1\to N_2$ a diffeomorphism we write $g_*\colon SN_1 \to S N_2$ to be the renormalized derivative map $$g_*(x,v) = \left(g(x), \frac{1}{ \|D_xg v\| }D_xg(v)\right).$$

In what follows, we consider    $C^{1+\beta}$ diffeomorphism $f\colon \R^k\to \R^k$ with uniform estimates: namely, viewing    $Df$ as a map from  $\R^k$ to the space of linear maps we assume  $ \sup _{x\in \R^k } \|D_x f\|<\infty$ and  that $Df$ is  H\"older continuous with $$\Hol(Df) := \sup_{x \neq y
 } \left\{\dfrac{ \| D_xf- D_y f \|} {d(x,y) ^\beta}\right\}
  <\infty .$$   

%
Given submanifolds $N_1$ and $N_2$ and a diffeomorphism $h\colon N_1\to N_2$ then, 
as  the linear maps $D_{x} h$ and $D_y h$ have different domains for $x\neq y\in N_1$, we define the H\"older variation of $Dh$ and $h_*$ as functions between metric spaces.  
Assuming $N_1$ has bounded diameter, define the   $\beta$-H\"older variation of $Dh \colon TN_1\to TN_2$ to be   $$\hol (D h) := \sup_{\substack{(x,v)\neq (y,u) \in SN_1  } 
 }\left\{ \dfrac{ d\left(Dh(x,v),Dh (y,u)\right) }{d\left((x,v), (y,u)\right) ^\beta}\right\} $$
 where, given $(x,v)$ and $(y,u)$ in $T\R^k$, we write 
	$$d\left((x,v), (y,u)\right) = \max\{d(x,y), d(v,u)\}.$$
Similarly define $\hol ( h_*)$.  
The    $C^{1+\beta}$-norm of $h$ is   $\max\{ \|h\|_{C^1}, \hol ( Dh ) \}$.

%
%
%

For the remainder, we let $f_n\colon \R^k\to \R^k$ be a sequence of uniformly $C^{1+\beta}$ diffeomorphisms with $f_n(0) = 0$ for each $n$. 
 Given $\epsilon'>0$ as above, we assume there is a $C_0>1$ so that for each $n$ 
 \begin{enumerate}
 \item $\|f_n-L_n\|_{{C^1}}\le \epsilon'$, and $\|f_n\inv -L_n\inv \|_{{C^1}}\le \epsilon'$;
%
%
%
\item $\hol (Df)<C_0$, and   $\hol (Df\inv)<C_0$,  
 \end{enumerate}
Note then that for some $C_1\ge C_0>1$ we have 
\begin{enumerate}[resume]
\item $\lochol ( ( f_n)_*)\le C_1 $\label{ji}, and   $\lochol  (( f_n\inv)_* )\le C_1 $;
\item $\| (D_x f_n^{\pm 1})_*\|_{C^1}\le C_1$ and $\|D_x f_n ^{\pm1 } \| \le C_1$ for every  $x$. 
\end{enumerate}
Here, $\lochol ( f_*) $ is the local H\"older variation  of $f_*\colon S\R^k\to S\R^k$ given by $$\lochol ( f_*) := \sup_{
0<d\left((x,v), (y,u)\right) \le 1 
 }\left\{ \dfrac{ d\left(f_*(x,v),f_* (y,u)\right) }{d\left((x,v), (y,u)\right) ^\beta}\right\} .$$  
  As it follows from all applications, we may moreover assume     that $f_n(y) = L_n(y)$ for all $y$ with $\|y\|\ge 1$.  
 
%
 
%
%

From  the graph transform method, we may construct for the sequence $f_n$ pseudo-stable manifolds through every point of $\R^n$.  
(See \cite{MR0501173} and discussion in \cite[Section 3]{MR2630044} for more details.) In particular, we have the following.  
\begin{proposition}\label{prop:HPS}
There exists a $\beta' >\bar \theta$ and $\beta''>0$ so that for every sufficiently small $\epsilon'>0$ and every $C_1>1$ as above there is a  $\hat C>0$ 
such that for 
every $n\in \Z$, $\star= \{u,c,s,cu,cs\}$, and $x\in \R^k$ there are manifolds $W^\star_n(x)$  containing $x$ with
\begin{enumerate}
\item $e^{\hat\kappa_n} d(x,y) \le d(f_n(x), f_n(y)) \le e^{\hat\eta_n} d(x,y) $ for $y\in W^u_n(x)$;
\item $e^{ \gamma_n} d(x,y) \le d(f_n(x), f_n(y)) \le e^{\hat\eta_n} d(x,y) $ for $y\in W^{cu}_n(x)$;
\item $e^{ \gamma _n} d(x,y) \le d(f_n(x), f_n(y)) \le e^{\hat\gamma _n} d(x,y)$ for $y\in W^c_n(x)$;
\item $e^{  \eta_n } d(x,y) \le d(f_n(x), f_n(y)) \le e^{\hat\gamma_n } d(x,y)$ for $y\in W^{cs}_n(x)$;
\item $e^{  \eta_n} d(x,y) \le d(f_n(x), f_n(y)) \le e^{  \kappa_n} d(x,y)$ for $y\in W^s_n(x)$;
\item $f_n (W^\star_n(x)) = W^\star_{n+1} (f_n(x))$ for every $x\in \R^k$ and $n\in \Z$.
\item If $y\in W^\star_n(x)$ then $W^\star_n(x) = W^\star_n(y)$.  In particular, the partition into $W^\star_n$-manifolds foliates $\R^k$; moreover, the partition into $W^s_n$-manifolds subfoliates each $W^{cs}_n(x)$.  
\item Each $W^\star_n(x)$ is the graph of a $C^{1+\text{H\"older}}$ function $G^\star _n(x)\colon \R^\star\to(\R^\star)^\perp$ with $\| D_uG^\star _n(x)\|\le \frac 1 3$  for all $u\in \R^\star$ and  
\begin{enumerate}
	\item $\Hol[\beta]( D  G^\star _n(x)) \le \hat C$ for $\star = s$;

	\item $\Hol[\beta']( D  G^\star _n(x)) \le \hat C$ 	for $\star = cs, c, cu$;
		\item $\Hol[\beta'']( D  G^\star _n(x)) \le \hat C$ for $\star = u$.
\end{enumerate}
Moreover, the functions $G_n^\star (x)$ depend continuously on $x$.  
\end{enumerate}
\end{proposition}
For a discussion of $C^{1+\text{H\"older}}$-regularity, see \cite[Section 6]{MR983869}.  Note in particular, that each $W^\star_n(x)$ is uniformly $C^{1+\text{H\"older}}$-embedded. 
We write $$W^\star_n(x,R) := \{ y\in W_n^\star (x): d(x,y)<R\}.$$

Write $E^\star_n(x):= T_xW^\star_n(x)$.  
From our choice of $\bar \theta>0$ satisfying \eqref{eq:holder}, it follows (for example from the $C^r$-section theorem \cite[page 30]{MR0501173}; see also discussion in \cite{MR2630044}) 
that the tangent spaces $E^\star_n(x)$ are H\"older continuous with exponent $\bar \theta$ and constant uniform in $x\in \R$ and $n\in \Z$.



%
\subsection{$C^{1+\text{H\"older}}$ holonomies inside center-stable manifolds }
Fix $R>0$.    Let $p\in \R^k$ and given $n\in \Z$ consider $q\in W^s_n(x,R)$.
Let $\hat D_1 $ and $\hat D_2$ be two  uniformly $C^{1+\beta'}$ embedded, $\dim(\R^{cu})$-dimensional manifolds  without boundary and with $p\in \hat D_1$ and $q\in \hat D_2$.  We assume the diameter of each $D_i$ is less than $1$ and that each subspace $T_x\hat D_i$ is sufficiently transverse to $\R^s$: given $v= T_x \hat D_i$ with $v= v^s + v^{cu}$ we have $\|v^{cu}\|\ge 3 \|v^s\|$.  
   Let $D_1= W^{cs}_n(p)\cap \hat D_1$  and $D_2= W^{cs}_n(q)\cap \hat D_2$.    
Given $x\in D_1$ let $h_{D_1, D_2}(x)$ denote the unique point $y$ in $D_2$ with $y\in W^s_n(x)$ if such a point exists.  Note that the domain and codomain of $h_{D_1, D_2}$ are open subsets of $D_1$ and $D_2$.  By  restriction of domain and codomain we may assume $h_{D_1, D_2}\colon D_1 \to D_2$ is a homeomorphism.  


Our main result is the following.  
\begin{theorem}\label{thm:difflholonomies}
The map $h_{D_1, D_2}$ is a $C^{1+\hat \alpha}$ diffeomorphism for some $\hat \alpha>0$.  Moreover, the 
 $C^{1+\hat \alpha}$-norm of $  h_{D_1, D_2}$    is uniform in all choices of $D_1$ and $D_2$ as above.  
\end{theorem}

In particular, we have 
\begin{corollary}\label{thm:lipholonomies}
The map $h_{D_1, D_2}$ is bi-Lipschitz with  Lipschitz constants uniform in all choices of  $D_1$ and $D_2$ as above.  
\end{corollary}

Recall that each $W^c_n(x)$ is a uniformly $C^{1+\beta'}$-embedded manifold and intersects $W^s_n(y)$ for every $y\in W^{cs}_n(x)$.  Moreover each $E^{cu}_n(y)$ is uniformly-transverse to both $\R^s$ and $E^s_n(y)$.  
As explained below, it suffices to prove Theorem \ref{thm:difflholonomies} for the distinguished family of transversals to $W^s_n$ given by  the family center manifolds.  Given $n\in \Z$, $p\in \R^k$, and $q\in W^s_n(p)$ we write  $h^s_{p,q,n}\colon W^c_n(p)\to W^c_n(q)$ for the stable holonomy map between center manifolds.   More precisely, 
$$h^s_{p,q,n}(z) = W^{cu}_n(q) \cap W^{s}_n(z).$$ As both $\{W _n^{s}(x): x\in W^{cs}_n(p)\}$ and $\{W _n^{c}(x): x\in W^{cs}_n(p)\}$ subfoliate  $ W^{cs}_n(p)$, it follows that $h^s_{p,q,n}(z) \in W^{c}_n(q).$
Moreover, by the global transverseness of the manifolds the maps $h^s_{p,q,n}$ have domain    all of $W^c_n(p)$. 

Note that $1<e^{\alpha \epsilon}<e^{\alpha(\hat \gamma_n - \gamma_n)}$. For remainder,  fix $0<\delta<1 $ so that for all $n\in \Z$  we have
 \begin{equation}\label{eq:delta} 1+ e^{-(\hat \gamma_n-  \gamma_n)} C_1 \delta^{\beta-\hat \beta} \le e^{\alpha (\hat \gamma_n-  \gamma_n)}.\end{equation}

Given $p,q,n$ with $q\in W^s_n(p)$ define $$\rho(p,q,n):= \sup \{d(x, h^s_{p,q,n}(x))\mid x\in 
W^c_n(p,1)\}.$$
Take $0<\rho_0<1$ so that  
\begin{equation} \label{llklklklkl}(3 C_2 C_1 +1)^{\bar \theta \inv} \rho_0\le \delta \end{equation} where $C_2\ge1$ is a constant to be defined in Section \ref{sec:aoorix} below. 
Fix $0<R_0<1$ for the remainder so that for all $n\in \Z$ ,  $p\in \R^k$ and $q\in W^s_n(p, R_0)$ we have $$\rho(p,q,n) \le \rho_0.$$
With the above choices, 
we prove    a special case of Theorem \ref{thm:difflholonomies}.
\begin{theorem} \label{thm:difflholonomiesspecial}\label{thm:difflholonomiesc}
Let $p\in \R^k$, $q\in W^s_n(p, R_0)$.  Then the  holonomy map $h^s_{p,q,n}\colon W^c_n(p,1) \to W^c_n(q)$ is  a  $C^{1+ \hat \alpha }$-diffeomorphism.  Moreover the   $C^{1+\hat \alpha }$-norm of $h^u_{p,q,n}$ is uniform across the   choice of $p$, $q$ and $n$.    
\end{theorem}

Taking $\hat \alpha <\beta'$, appealing to a theorem of Journ\'e \cite{MR1028737} (see also related discussions in \cite[Section 6]{MR1432307})
it follows that the leaves of the partition $\{W^s_n(x), x\in \R^k\}$ restrict to a $C^{1+\hat \alpha }$-foliation inside each $W^{cs}_n(p)$.  Theorem \ref{thm:difflholonomies} then follows.  
In particular, given arbitrary transversals $D_1  $ and $D_2$ to  $\{W^s_n(x), x\in W^{cs}_n(p) \}$ inside $W^{cs}_n(p)$ as above, it follows that the holonomy map  $h^s_{D_1, D_2}$ is uniformly $C^{1+\hat \alpha }$ on its domain.


\section{Proof of Theorem \ref{thm:difflholonomiesspecial}}\label{Sec:2}
\subsection{Approximate holonomies and related notation}\label{sec:aoorix}

Given $n\in \Z$ and arbitrary $p,q\in \R^k$ with $q\in W^s_n(p, 1)$ we assume there exists a uniformly $C^{1+\beta'} $ initial approximation $$\pi_{p,q,n} \colon W^c_n(p, 1)\to W^c_n(q)$$ to  the stable holonomy map $h_{p,q,n}^s$ 
with the following properties:
there is a     constant $C_2>0$ 
 so that for every $n\in \Z$,  $ p\in \R^k$, and $q\in W^s(p,1)$ we have 
\begin{enumerate}
	\item \label{item1}$d(\pi_{p,q,n} (p), q)\le C_2d(p,q)$; 
	\item \label{item2}$d \left((\pi_{p,q,n} )_*(v), v\right)\le C_2d(p,q)^{\bar\theta}$ for all $v\in S_pW^c_n(p)$;
	\item \label{item2'}$\left|\| D\pi_{p,q,n} \|  - 1\right|\le C_2d(p,q)^{\bar \theta}$;  	
	\item \label{item3}  if $p'\in W^c_n(p)$ and $q'\in W^s_n(p', 1)\cap W^c_n(q)$ then 
	$\pi_{p,q,n} $ and $\pi_{p',q',n} $ coincide on $W^c_n(p,1)\cap W^c_n(p',1)$. 
\end{enumerate}

For instance, we may define such a system of approximating maps $\{\pi_{p,q,n}\}$   by linear projection: for $z\in W^c_n(p, 1)$ define $\pi_{p,q,n} ( z)$ to be  the unique point of intersection of $W^c_n(q)$ and $ z + \R^u \oplus \R^s$.  One may verify the above properties hold for this choice of $\pi_{p,q,n}$.
\subsubsection{Additional notation}It is enough to prove Theorem \ref{thm:difflholonomiesc} in the case that $n=0$.
For the remainder, we  fix $p$ and $q$ in $ \R^k$ with  $q\in W^s_0(p, R_0)$ as in Theorem \ref{thm:difflholonomiesc}.  Write  $h:=h_{p,q,0}^s$.

%

Given $n, j\in \Z$ 
 \begin{itemize} 
\item $f^{(j)}_n:= \mathrm{id}$, $j=0$;
\item $f^{(j)}_n:= f_{n+j-1} \circ \dots \circ f_n$, $j>0$;
\item $f^{(j)}_n:= f_{n+j}\inv  \circ \dots \circ f_{n-1}\inv$, $j<0$;
\item for $z\in \R^k$, write  $z_n = f^{(n)}_0(z)$;
\item write  $D_n\subset W^c_n(p_n):= f^{(n)}_0(W^c_0(p,1))$; 
\item let $\displaystyle \kappa^{(n)}_\ell =\begin{cases}\kappa_{\ell+n-1}+  \dots + \kappa_\ell & n>0\\
0 &n= 0\\
-\kappa_{\ell+n}-  \dots  - \kappa_{\ell-1} & n<0;
\end{cases}$
\item similarly define $\hat \kappa ^{(n)}_\ell, 
\gamma ^{(n)}_\ell,$ and $ \hat \gamma ^{(n)}_\ell $.
\end{itemize}

Note that if $x\in W^c_0(p,1) = D_0$ 
 and $y= h(x)\in W^c_0(q)$ then for all $n\ge 0$ we have 
$$d(x_{n},y_{n})\le e^{\kappa _0^{(n)}} d(x,y)\le e^{\kappa _0^{(n)}} { \rho(p,q,0)<\rho_0.}$$ 
It follows that $\pi_{x_{n},y_{n},n}$ is defined.  
By property (\ref{item3}) of the approximate holonomy maps $\pi_{p_{n}, q_{n}, n}$ it follows that the collection of maps $\{\pi_{x_{n},y_{n},n}: x_n\in D_{n}\}$ coincide with the restriction of a   single approximation which we denote by $\pi_{n} \colon D_{n}\to W^c_{n}(q_{n})$ for the remainder.  Note that $\pi_{n} \colon D_{n}\to W^c_{n}(q_{n})$  has all the properties enumerated above.  

\subsubsection{Approximate holonomies}
For $n\ge 0$ we define   $h_{n}\colon W^c_0(p,1) \to W^c_0(q)$ to be successive approximations to $h$ given by 
	$$h_{n} \colon x\mapsto f^{(-n)}_{n} \left( \pi_{n} (x_{n}))\right)=f^{(-n)}_{n} \left( \pi_{n} (f^{(n)}_0(x))\right).$$
Note that each $h_{n} $ is a $C^{1+\beta'}$ diffeomorphisms onto its image.  Although the  $(1+\beta')$-norms of the $h_{n}$ may not be controlled, 
Theorem \ref{thm:difflholonomies} follows by showing that $h_{n} $ converge to $ h \colon W^c_0(p,1) \to W^c_0(q)$  in the $C^1$ topology.  We then show $Dh\colon SW^c_0(p,1) \to TW^c_0(q)$ is H\"older continuous with uniform estimates.  

\subsection{An auxiliary lemma.}

Given  $\xi= (x, v)$ and $\zeta= (y, u)$ in $S \R^k$ recall we write $d(\xi,\zeta) = \sup \{d(x,y), d(v,u)\}$.  
Given $\xi =(x,v)\in S\R^k $ 
we write  $$\xi_{n}=(x_{n}, v_{n}):= (f^{(n)}_0)_*(\xi) \in S\R^k.$$
If $\zeta= (y, w)$ similarly write $\zeta_{n}:= (y_{n}, w_{n})$.

Recall the $\delta>0$,  $\alpha $, and $\hat \beta$  fixed above.

\begin{lemma}\label{lem:1,1}
Given  $x\in \R^k$, $\xi=(x,v)$, $\zeta=(y,w) \in S W^c_0(x) $, $0\le r\le \delta$, and $n\ge 0$, suppose
		that  $d(x_{n}, y_{n})\le  r e^{\kappa_0^{(n)}}$, $\DPP(\xi_{n}, \zeta_{n})\le  r^{\bar \theta }e^{\kappa_0^{(n)} \theta} $, 
	and  for all $0\le k\le n$ that $$ d(x_{k}, y_{k})\le \delta.$$ Then,  for all $0\le k\le n $,
	$$d(x_{k}, y_{k})\le r  e^{  \kappa_0^{(n)} -    \gamma_{k}^{(n-k)} }  \text{ and } \DPP(\xi_{k}, \zeta_{k})\le  r^{\bar \theta } e^{  \kappa_0^{(n)}  \theta   + (1+\alpha)(\hat\gamma_{k}^{(n-k)}- \gamma_{k}^{(n-k)}) } .$$
	In particular, 
	$$d(\xi, \zeta) \le   r^{\bar \theta } e^{ \kappa_0^{(n)}  \theta    + (1+\alpha) \left(\hat\gamma_0^{(n)}- \gamma_0^{(n)}\right)} .$$

		

\end{lemma}

\begin{proof}			
 For the final assertion, note that 
 $$r e^{\kappa_{0}^{(n)} - \gamma_{0}^{(n)}}\le  r^{\bar \theta } e^{\kappa_{0}^{(n)}  \theta    + (1+\alpha)(\hat \gamma_{0}^{(n)}- \gamma_{0}^{(n)})}$$ follows from \eqref{eq:holder} as  $  \theta\le \frac {\kappa_n -\gamma_n}{\eta_n}\le \frac {\kappa_n -\gamma_n}{\kappa_n} $ holds for all $n$.     

We prove  the first two assertions by backwards induction on $k$ starting with $k=n$.     
We clearly have 
$$d(x_{(k-1)}, y_{(k-1)}) \le e^{-\gamma_{k-1}} d(x_{k}, y_{k}) \le  r e^{\kappa _{0}^{(n)} - \gamma _{k-1}^{(n-(k-1))} }  .$$
Moreover, we have 
\begin{align*}
d(\xi_{(k-1) }, &\zeta_{(k-1)})
= d\left((f_{k}^{(-1)})_* (x_{k}, v_{k}) , (f_{k}^{(-1)})_* (y_{k}, w_{k})\right) \\
&\le  d\left((f_{k}^{(-1)})_* (x_{k}, v_{k}) , (f_{k}^{(-1)})_* (x_{k}, w_{k})\right) 
\\ &  \quad\quad\quad \quad 
+
  d\left((f_{k}^{(-1)})_* (x_{k}, w_{k}) , (f_{k}^{(-1)})_* (y_{k}, w_{k})\right) \\
&\le   e^{(\hat\gamma_{k-1}-\gamma_{k-1})}\DPP(v_{k}, w_{k}) +
C_1d(x_{k} , y_{k})^\beta \\
&\le  e^{(\hat\gamma_{k-1}-\gamma_{k-1})}\left
(\DPP(v_{k}, w_{k}) +
e^{-(\hat\gamma_{k-1}-\gamma_{k-1})}C_1 d(x_{k} , y_{k})^\beta \right )\\
&\le  e^{(\hat \gamma_{k-1}-\gamma_{k-1})}\left
(1 +
e^{- (\hat \gamma_{k-1}-\gamma_{k-1})}C_1d(x_{k} , y_{k})^{\beta-\hat \beta}  \right) \\&\quad\quad \cdot
\max \{ d(x_{k} , y_{k})^{\hat \beta} ,\DPP(v_{k}, w_{k})\}  \\
&\le  e^{ (1+\alpha) (\hat\gamma_{k-1}- \gamma_{k-1})}
\max \{ r^{\hat \beta} e^{ \hat \beta \kappa _0^{(n)} - \hat \beta \gamma_{k}^{ (n- k)} } , r^{\bar \theta } e^{ \kappa _0^{(n)}  \theta     + (1+\alpha)(\hat \gamma_{k}^{ (n- k)}- \gamma_{k}^{ (n- k)}) } \}.  \notag 
\end{align*}
The last line follows from induction hypothesis and the choice of  $\delta>0$ in \eqref{eq:delta}. 


From \eqref{eq:coarsbunch} and \eqref{eq:3} we   have  
we have  
\begin{align*}
 \hat \beta \kappa _0^{(n)} - \hat \beta \gamma_{k}^{ (n- k)}  & 
= \hat \beta \kappa_{0}^{(k)}  + \hat \beta \kappa_{k}^{(n-k)} -  \hat \beta\gamma _{k}^{ (n- k)} \\
&\le \hat \beta \kappa_{0}^{(k)}+   \theta   \kappa_{k}^{(n-k)}  \\
&\le  \theta   \kappa_{0}^{(n)} \\
&\le \theta   \kappa _0^{(n)}     + (1+\alpha)(\hat \gamma_{k}^{ (n- k)}- \gamma_{k}^{ (n- k)}).
\end{align*}
Hence  
$$r^{\hat \beta} e^{ \hat \beta \kappa _0^{(n)} - \hat \beta \gamma_{k}^{ (n- k)} } \le   r^{\bar \theta }  e^{ \kappa _0^{(n)}  \theta     + (1+\alpha)(\hat \gamma_{k}^{ (n- k)}- \gamma_{k}^{ (n- k)}) } 
$$
and the result follows. \end{proof}

\subsection{Step 1: $C^0$ convergence} 
We have
\begin{lemma}  $h_n \to h$ uniformly on $W^c_0(p,1)$. 
\end{lemma}
\begin{proof}    
First  (by invariance of $W^s_n$-manifolds) we have 
$f^{(-n)} _n\circ h^s_{p_{n}, q_{n},n} \circ f^{(n)} _0= h^s_{p,q,0}$. For $x\in W^c_0(p,1)$
$$d\left(x_{n}, f^{(n)}_0(h(x))\right)= d(x_{n}, h^s_{p_{n}, q_{n},n}(x_{{n}}))\le  e^{\kappa_0^{(n)}}  \rho(p,q,0)\le e^{\kappa_0^{(n)}} 
.$$
By property \ref{item1} of the maps $\pi_{n}$, 
\begin{align*}
d(h_n(x), h(x))
	&= d\left(f^{(-n)}_0( \pi_{n} (x_{n})), f^{(-n)}_n(h^s_{p_{n}, q_{n},n}(x_{n}))\right)\\
	&\le e^{ -\gamma_{0}^{(n)} } d\left(\pi_{n} (x_{n}), h^s_{p_{n}, q_{n},n}(x_{n})\right)\\	
	&\le C_2 e^{{\kappa_0^{(n)}} -\gamma_{0}^{(  n)} }. \qedhere	
\end{align*}
\end{proof}

\subsection{Step 2: Convergence of the projectivized derivative.} 
Consider now the projectivized derivatives $(h_n)_*\colon S W^c_0(p,1) \to S W^c_0(q).$ 
 We   show  that the sequence $(h_n)_*$ is Cauchy.  Set 
$$L_1= C_2 + \sum_{n=0}^\infty 3C_2 C_1 e^{\omega n} .$$ 

\begin{lemma} \label{lem:5}
The family of maps $(h_n)_*\colon S W^c_0(p) \to S W^c_0(q) $ is uniformly Cauchy on $S W^c_0(p,1)$.  

Moreover, defining 
$h_*\colon S W^c_0(p,1) \to S W^c_0(q)$ to be the limit $h_*= \lim_{n\to \infty} (h_n)_*$, 
for $(x,v)\in SW^c_0(p,1)$ we have 
$$d((x,v), h_* (x,v))\le  L_1 d(x, h(x) ) ^{\bar \theta}.  $$
\end{lemma}
\begin{proof} 
With $\xi= (x,v)$ let  $y = h(x)$ and $\zeta^n =(y^n, w^n)= (h_n)_*(\xi)$.  

 We have 
 $$d(y^n, y^{n+1}) 	= d\left(  f^{(-n)}_{n} (\pi_{n} (x_{n})),   f^{(-n-1)}_{(n+1)}(\pi_{n+1}(x_{n+1}))\right)$$
 and 
 \begin{align*}
 d\left(  \pi_{n} (x_{n}), f_{n+1}^{(-1)}( \pi_{n+1}(x_{n+1}))\right)
  &\le   d\left( x_{n},  \pi_{n} (x_{n})\right)+ \left( f_{n+1}^{(-1)}(x_{n+1} ), f_{n+1}^{(-1)}(\pi_{n+1}(x_{n+1}))\right)\\
    &\le  C_2 e^{\kappa _0^{(n)}}d(x,y)  + C_1 C_2 e^{\kappa_0^{(n+1)}}d(x,y) \\
&\le 2C_1 C_2 e^{\kappa_0 ^{(n)} }d(x,y).
 \end{align*}
 
 Similarly, %
\begin{align*}
 d(\zeta^n, \zeta^{n+1}) &= 
%
  d\left( (f^{(-n)}_{n})_* (\pi_{n})_* (\xi_{n}),   (f^{(-n-1)}_{n+1})_* (\pi_{n+1})_* (\xi_{n+1})\right)
\end{align*}
 and  
 \begin{align*}
d( (\pi_{n})_* (\xi_{n}),  &  (f^{(-1)}_{n+1})_* (\pi_{n+1})_* (\xi_{n+1})) \\
&\le 
d\left(  \xi_{n}, (\pi_{n})_* (\xi_{n})  \right) + d \left((f_{n+1}^{(-1)})_*   \xi_{n+1},   ( f_{n+1}^{(-1)})_* (\pi_{n+1})_* (\xi_{n+1})\right).\end{align*}	
With $(\pi_{n+1})_* (x_{n+1}, v_{n+1})= (y',w')$ we have 
 \begin{align*}
d  ((f_{n+1}^{(-1)})_* &   \xi_{n+1},    ( f_{n+1}^{(-1)})_* (\pi_{n+1})_* (\xi_{n+1}) )=
d  ((f_{n+1}^{(-1)})_*   (x_{n+1}, v_{n+1}) ,   ( f_{n+1}^{(-1)})_* ( y',w')  )\\
&\le d  ((f_{n+1}^{(-1)})_*   (x_{n+1}, v_{n+1}) ,   ( f_{n+1}^{(-1)})_* ( x_{n+1},w')  )+ 
d  ((f_{n+1}^{(-1)})_*   (x_{n+1}, w') ,   ( f_{n+1}^{(-1)})_*(  y',w')  )\\
&\le C_1 d( v_{n+1}, w') + C_1 d(x_{n+1}, y')^\beta\\
&\le C_1 C_2 e^{\kappa_0 ^{(n+1)} \bar \theta} d(x,y)^{\bar \theta}  + C_1 C_2^\beta e^{\beta \kappa_0 ^{(n+1)} } d(x,y)^\beta\\
&\le 2 C_1 C_2  e^{\bar \theta \kappa_0 ^{(n+1)} }d(x,y)^{\bar \theta  }
\end{align*}	
whence
 \begin{align*}
d( (\pi_{n})_* (\xi_{n}),  &  (f^{(-1)}_{n+1})_* (\pi_{n+1})_* (\xi_{n+1})) \\
&\le C_2 e^{\bar  \theta \kappa_0^{( n)}}d(x,y)^{\bar \theta} 
+  2C_1 C_2  e^{ \bar \theta   \kappa_0^{ (n+1)}} d(x,y)^{\bar \theta}\\
&\le 3C_2 C_1  e^{\bar \theta  \kappa _0^{(n)}} d(x,y)^{\bar \theta }.
\end{align*}


%
From  Lemma \ref{lem:1,1}  (with $r= (3C_1C_2)^{\bar \theta \inv}d(x,y) $) and  the choice of $\rho_0$ it follows that 
\begin{align}\label{eqMMM}
d(\zeta^n, \zeta^{n+1} ) \le   e^{ \kappa_0^{(n)}  \theta   + (1+\alpha)(\hat \gamma_0^{(n)}- \gamma_0^{(n)}) } 3C_2 C_1  d(x,y)^{\bar \theta}. 
  \end{align}
It follows that $(h_n)_* $ is uniformly Cauchy on $SW^c_0(p,1)$ .

%

Moreover, for any $\xi=(x,v)\in SW^c_0(p,1)$ we have  
\begin{align*}
d(\xi, h_* (\xi) ) &\le d(\xi, (\pi_0)_* \xi) + \sum_{n=0}^\infty d(\zeta^n, \zeta^{n+1} ) \\
&\le C_2d(x,y)^{\bar \theta } + \sum_{n=0}^\infty e^{ \kappa_0^{(n)}  \theta   + (1+\alpha)(\hat \gamma_0^{(n)}- \gamma_0^{(n)}) } 3C_2 C_1   d(x,y)^{\bar \theta}.  
\end{align*}
Thus, with $L_1$ as above, 
for any $\xi\in SW^c_0(p,1)$ we have 
\begin{align*}
&d(\xi, h_* (\xi) ) \le  L _1d(x,y) ^{\bar \theta }.  \qedhere
\end{align*}\end{proof}

Note that the convergence of the projectivized derivative of the stable holonomies  in Lemma \ref{lem:5} is independent of the choice of $p,q\in \R^k$ or $n\in \Z$ in Theorem \ref{thm:difflholonomies}. 
Thus for all $n\in \Z$, $p'\in \R^k$, and $q'\in W^s_{n} (p', R_0)$ let $(h^s_{p',q',n})_*$ denote the projectivized derivative constructed as above. 
We have for all $\xi'= (x',v')\in SW^c_n(p',1)$ that $$d(\xi,(h^s_{p',q',n})_* (\xi) ) \le  L _1d(x',h^s_{p',q',n}(x')) ^{\bar \theta  }.$$ 
 To show that the holonomies are $C^1$, we next show that each  $(h^s_{p',q',n})_*$ coincides  with  the projectivization of a continuous $D h^u_{p',q',n}\colon TW^c_n(p',1)\to TW^c_n(q')$.  

\subsection{Step 3: The sequence of maps $ Dh_n$ is uniformly Cauchy.}
We return to the notation in Step 2.  In particular, we recall our distinguished $p,q\in \R^k$ and the maps $h_n$ approximating $h= h^s_{p,q,0}$. 

We first derive two simple distortion estimates.  
Given $\xi=(x,v)\in S W^c_0(p,1)$ with  $\xi_{n}:= (f^{(n)}_0)_*(\xi) $  let $y=h(x)$,  $\zeta =(y,w) := h_*(\xi)$, 
$\zeta_{n}=(y_{n}, w_{n}) = (f_0 ^{(n)})_* (\zeta) = (h^s_{p_{n},q_{n}, n})_* (\xi_{n}),$ 
	and $\hat \zeta^n= (h_n)_*(\xi)$.  Write $\hat \zeta ^n_i = (f^{(i)}_0)_*\hat \zeta ^n$. Then 
	$$\hat \zeta^n_{n }:= (f_0^{(n)})_*(\hat \zeta ^n)= (\hat y^n_{n}, \hat w_{n})  = (\pi_{n})_* (\xi_{n}).$$
We have \begin{equation}\label{eqKKK} d(y_{n}, \hat y^n_{n}) \le (1+ C_2 ) e^{\kappa _0^{(n)} }d(x,y).\end{equation}  
From Lemma \ref{lem:5} we have 
$$d(\xi_{n}, \zeta_{n})\le L _1e^{\bar \theta \kappa _0^{(n)}} d(x,y)^{\bar \theta }. $$  
Moreover $$d(\xi_{n}, \hat \zeta^n_{n}) \le C_2 e^{\bar \theta \kappa _0^{(n)}} d(x,y)^{\bar \theta }. $$
Hence, 
$$d(\zeta_{n}, \hat \zeta^n_{n})\le (C_2 +L_1 ) e^{ \bar \theta \kappa _0^{(n)} }.$$

Let $n_0$ be such that $(C_2 +L _1) e^{ \bar \theta \kappa _0^{(n_0)}}\le e^{   \theta     \kappa _0^{(n_0)}}	\delta^{ \bar \theta}$
so that for 
$n\ge n_0$ we have 
\begin{equation}\label{eqLLL}d(\zeta_{n}, \hat \zeta^n_{n}) \le  
e^{     \theta     \kappa _0^{(n)} }	\delta^{\bar \theta}. \end{equation}
 


Given $\xi= (x,v) \in T\R^k$ define $\|\xi\| = \|v\|$.  
For each $i$,  the map $S \R^k \to \R$ given by $\zeta\to \log \|Df_i(\zeta)\|$  is $\beta$-H\"older on $S\R^k$ with uniform constant $C_3$.  Recall our $\hat \theta <\theta $ and $\omega$ with 
$\hat \omega = \sup\kappa_n\hat \hat  \theta + (1+\alpha)(\hat\gamma _n-  \gamma_n)<0$.  
Let $$K_1:= 
\sum _{k=0}^{\infty}  C_3 \left(\delta^{\bar \theta} e^{  \hat \omega k } 
 \right) ^\beta \\
.$$

\begin{lemma}\label{lem:bounded}
For all $n\ge n_0$ 
$$\exp (- K_1e^{  \beta  (\theta -\hat \theta  )  \kappa_0^{(n)} }) \le \dfrac{\prod_{i=0} ^{n-1} \|Df_{i}  \hat \zeta^n_{i}\|}{\prod_{i=0} ^{n-1} \|Df_{i}  \zeta_{i}\|}
\le   \exp(K_1e^{  \beta  (\theta -\hat \theta  )  \kappa_0^{(n)} }).$$
\end{lemma}
(Note in particular that the middle ratio goes to $1$ as $n\to \infty$.) 
\begin{proof}
Recalling Lemma \ref{lem:1,1}  
(with $r=\delta$ and estimates \eqref{eqKKK} and \eqref{eqLLL}), for $n\ge n_0$ we have 
\begin{align*}
\left| \log \dfrac{\prod_{i=0} ^{n-1} \|Df_{i}  \hat \zeta^n_{i}\|}{\prod_{i=0} ^{n-1} \|Df_{i}  \zeta_{i}\|}
\right|
&= \left |
	\sum _{i=0}^{n-1}  \log  \|Df_{i}  \hat \zeta^n_{i}\| -  \log \|Df_{i}  \zeta_{i}\| 
\right |\\
&\le  
	\sum _{i=0}^{n-1} C_3 d (\hat \zeta^n_{i},\zeta_{i})^\beta
\\
&\le  
	\sum _{i=0}^{n-1} C_3 \left(\delta^{\bar \theta} e^{ \kappa_{0}^{(n)}  \theta    + (1+\alpha)(\hat \gamma_i^{(n-i)}-  \gamma_i^{(n-i)})} \right) ^\beta \\
&\le  e^{  \beta \kappa_0^{(n)}(\theta  -\hat \theta ) } K_1.		\qedhere
\end{align*}
\end{proof}

%
%
Similarly, letting $$ K_2 = \exp
\left(\sum _{i=1}^{\infty } C_3     L_1^\beta  e^{\bar \kappa  \bar \theta\beta i} \right) $$
we have 
\begin{lemma}\label{lem:thisone}
For all $n>0$ 
$$ K_2\inv \le 
\dfrac{\prod_{i=0} ^{n-1} \|Df_{i}    \xi_{i}\|}{\prod_{i=0} ^{n-1} \|Df_{i}  \zeta_{i}\|}
\le    K _2 $$
\end{lemma}

\begin{proof} From Lemma \ref{lem:5} we obtain 
\begin{align*}
\left| \log \dfrac{\prod_{i=0} ^{n-1} \|Df_{i}    \xi_{i}\|}{\prod_{i=0} ^{n-1} \|Df_{i}  \zeta_{i}\|}
\right|
&= \left |
	\sum _{i=0}^{n-1}  \log    \|Df_{i}    \xi_{i}\| - \log \|Df_{i}  \zeta_{i}\|  \| 
\right |\\
&\le  
	\sum _{i=0}^{n-1} C_3 d( \xi_{i},\zeta_{i})^\beta
\\
&\le  
	\sum _{i=0}^{n-1} C_3 \left[L_1 d(x_i, y_i)^{\bar \theta  } \right]^\beta
\\
&\le  
	\sum _{i=0}^{n-1} C_3  L_1^\beta  \left(e^{\kappa_0^{(i)} \bar \theta} \right) ^\beta d(x,y)^\beta \\
&\le \log(  K_2)  .		\qedhere
\end{align*}
\end{proof}

Approximate the derivatives $D h\colon TW^c_0(p) \to TW^c_0(q)$ by the bundle maps 
$$\Delta_{n} \colon  TW^c_0(p,1)\to TW^c_0(q)$$ defined as follows:  given $n\ge 0$ and $(x,v) \in    TW^c_0(p)$  let 
$$\Delta_{n} \colon (x,v)\mapsto Df^{(-n)}_{n}\left(\left(\| D f^{(n)}_0 (v)\|    (h^s_{p_{n},q_{n},n})_*((f^{(n)}_0)_*(x,\tfrac{v}{\|v\|})))\right)\right).$$
Note that $(\Delta_{n})_* = h_*$.   From Lemma \ref{lem:thisone}, with $h_*(x,v) = (y,w)$ we have    $$\|\Delta_{n} (x,v) \| = \frac{\|D_x f_0^{(n)} (v) \|}{\|D_y f_0^{(n)} (w) \|}\le K_2.$$ 
In particular, $\|\Delta_{n} \| $ is uniformly bounded.

Also,  given $\xi=(x,v)\in SW^c_0(p)$ we have 
$$\|Dh_n (\xi) \|= \|\Delta_n (\xi)\| \cdot \dfrac{\prod_{i=0} ^{n-1} \|D f_{i}     \zeta^n_{i}\|}{\prod_{i=0} ^{n-1} \|D f_{i} \hat   \zeta_{i}\|}
	 \cdot \| D\pi_{p_{n},q_{n},n}((f^{(n)}_0)_*(x,\tfrac{v}{\|v\|}))\|. $$
	It then follows from Lemma \ref{lem:bounded} that $$\|Dh_n (\xi)\|  - \|\Delta_n (\xi)\| \to 0$$ uniformly in $ \xi\in SW^c_0(p)$.
 Moreover, as the  the projectivization of $\Delta_n$ coincides with $(h^s_{p,q,0})_*$ we have
\begin{claim}\label{claim:salmi}
$\sup_{\xi\in SW^c_0(p,1) } \{\left\|Dh_n (\xi) - \Delta_n(\xi)\right\|\}\to 0$ as $n\to \infty$.  
\end{claim}

It follows that the sequence  $D h_n $ converges uniformly if and only if the sequence $\Delta_n$ converges uniformly.  

%

\begin{lemma}\label{lem:deltaconverge}
The sequence of maps $\Delta_n\colon SW^c_0(p,1) \to TW^c_0(q)$ is uniformly Cauchy.  
\end{lemma}
\begin{proof}
Given $\xi=(x,v) \in SW^c_0(p,1) $ with  $\xi_{n}= (x_{n}, v_{n})\in S D_{n}$, let  $\zeta =(y,w) = h_*(\xi)\in SW^c_0(q) $ and  $\zeta_{n}=(y_{n}, w_{n})= (h^u_{p_{n},q_{n},{n}})_{*}(\xi_{n}) $.
Observe  that both $\Delta_{n} (x,v)$ and $ \Delta_{n+1} (x, v) $ have footprint   $y$.  
Then 
\begin{align}
\| \Delta_{n} (x,v)& - \Delta_{n+1} (x, v) \| \notag\\
&= \frac{\| D_x f_0^{(n)} (v)\| }{ \| D_{y}f^{(n)}_0 ( {w})\| } \|D_{y_{n}} f_{n} (w_{n})\| \inv \\&\quad \quad  \cdot \left (   \|D_{y_{n}} (f_{n}) (w_{n})\| - \|D_{x_{n}} (f_{n}) (v_{n})\| 
\right )\notag \\
& \le   K_2 C_1
 \left( C_1 L _1e^{\bar  \theta \kappa _0^{(n)}}   + C_1  e^{\kappa _0^{(n)} \beta }\right) \notag
\\
&\le   K_2 C_1  2  C_1 L _1  e^{ \bar \theta \kappa _0^{(n)}}  \label{eq:lllllll}\\
&\le   K_2 C_1  2  C_1 L _1  e^{\bar \theta \bar \kappa n}.\notag \qedhere
\end{align}
\end{proof}
From Claim \ref{claim:salmi} and Lemma \ref{lem:deltaconverge} it follows that the  sequence of maps $D h_{ n} $ converges uniformly.  As $h_{ n}$ converges to $h$ we necessarily have that $h = h^s_{p,q,n}$ is differentiable and that $Dh_{ n}$ converges to $Dh$.  
Furthermore, 
 $\|Dh^s_{p,q,n}\|\le K_2$.  
     This completes the proof of  $C^1$ properties in Theorem \ref{thm:difflholonomiesspecial}.  
%
%


\subsection{Step 4: H\"older continuity of $Dh$}
We have the following claim.
\begin{claim}\label{claim:kkklk}There is a $c_0>0$ so that 
if $d(x_k,y_k)\le 1$ for all $0\le k\le n-1$ and $d((x,v), (y,w))\le 1$ then 
$$d((f_0^{(n)})_*(x,v), (f_0^{(n)})_*(y,w))\le e^{c_0 n}  d( (x,v),  (y,w))^\beta$$
\end{claim}
\begin{proof}
We have 
\begin{align*} d((f_0^{(n)})_*(x,v), & (f_0^{(n)})_*(y,w)) \le 
d((f_0^{(n)})_*(x,v), (f_0^{(n)})_*(x,w))+ d((f_0^{(n)})_*(x,w), (f_0^{(n)})_*(y,w))\\
&\le (C_1)^{2n} d(v,w) +  d((f_0^{(n)})_*(x,w), (f_0^{(n)})_*(y,w))\end{align*}

Also proceeding inductively we have 
\begin{align*} 
d((f_0^{(n)})_*(x,w), & (f_0^{(n)})_*(y,w)) 
= 
d((D_{x_{n-1}}f_{n-1}) _*  (D_{x}f_0^{(n-1)})_* (w) , (D_{y_{n-1}}f_{n-1}) _*  (D_{y}f_0^{(n-1)})_* (w) ) 
\\ &\le 
d((D_{x_{n-1}}f_{n-1}) _*  (D_{x}f_0^{(n-1)})_* (w) , (D_{y_{n-1}}f_{n-1}) _*  (D_{x}f_0^{(n-1)})_* (w) ) \\
&\quad \quad + 
d((D_{y_{n-1}}f_{n-1}) _*  (D_{x}f_0^{(n-1)})_* (w) , (D_{y_{n-1}}f_{n-1}) _*  (D_{y}f_0^{(n-1)})_* (w) ) 
\\ & \le 
C_1 d(x_{n-1}, y_{n-1}) ^\beta+  C_1  (n-1) C_1 ^{  (1+\beta)(n-1)} d(  x,y)^\beta
\\ &\le 
C_1 (C_1)^{n\beta} d(x, y) ^\beta+  C_1  (n-1) C_1 ^{ (1+\beta)(n-1)} d(  x,y)^\beta
\\&  \le 
n C_1 ^{  (1+\beta)n} d(  x,y)^\beta.
\end{align*}
Take $c_0 = (\log C_1 )(1+\beta) +   1$.  
\end{proof}

We now show that the maps $h_*\colon SW^c_0(p,1)\to SW^c_0(q)$ and $Dh\colon SW^c_0(p,1)\to TW^c_0(q)$ are H\"older continuous.  

Given $\epsilon>0$ take
\begin{itemize}
\item  $a_0= \max \left\{ 
 \sup\{ \frac{\kappa_n - \hat \gamma_n }{\kappa_n}\},
  \sup\{ \frac{\bar \theta \kappa_n - \beta \hat \gamma_n -\epsilon }{\beta  \kappa_n}\}  
  ,\sup  \{\frac { \kappa_n \bar \theta- c_0 -\epsilon}{ \bar  \theta  \kappa_n  \beta   }\}  \right \} >1$;
\item Set  $\bar \alpha = \min \left\{ 
 \inf \{ \frac{\kappa_n -  \gamma_n }{a_0 \kappa_n}\},  \inf \{ \frac{  \theta   \kappa _n + (1+\alpha)(\hat  \gamma _n - \gamma_n)}{\kappa_n a_0    } \} , \frac {\bar \theta} {a_0 } 
  \right \}.$  
Note $0<\bar\alpha <1$.  
\item Recall our fixed $\rho_0$  with $\rho(p,q,0)\le \rho_0 < \delta$ satisfying \eqref{llklklklkl} and let $\rho = \rho_0/(1+ 2 C_2)^{\theta\inv}.$

\end{itemize}
Let $q\in W^s_0(p)$ and as before let $h = h^s_{p,q,0}$.  Assume now that $d(p,q) \le \rho$.  
\begin{claim}
The function $h_*$ is $\bar \alpha$-H\"older with constant uniform in all parameters.  
\end{claim}

\begin{proof}

Given $n\in \N$ let 
$$r_n : = \rho e^{\kappa_0^{(n)}  a_0   } .$$
Consider any pair $\xi := (x,v)$ and $ \xi':= (x',v') $  in $SW^c_0(p, 1)$.  
It is enough to consider pairs that are sufficiently close so that   for some  $1\le n$:  
	$$d(x,x') \le  r_n ,   \DPP (v, v') \le (r_n)^{\bar \theta} ,   
\text{ and either $d(x,x')\ge r_{n+1}$ or $\DPP (v, v') \ge r_{n+1}^{\bar \theta}$}.$$
Let $\zeta= (y, w) = h_*(\xi)$ and $\zeta^n= (y^n, w^n) = (h_n)_*(\xi)$. Similarly define $\zeta'$ and $\zeta'^{ n}$.  
From \eqref{eqMMM} we have  
 $$d(\zeta, \zeta^n) \le e^{(\kappa_0^{(n)}  \theta   + (1+\alpha)(\hat \gamma_0^{(n)}- \gamma_0^{(n)})} L_1
 \le L_1 \rho^{-  \bar \alpha }  e^{\mu\bar \alpha a_0} r_{n+1}^{ \bar \alpha }$$  and 
$$d(\zeta', \zeta'^n)\le e^{(\kappa_0^{(n)}  \theta   + (1+\alpha)(\hat \gamma_0^{(n)}- \gamma_0^{(n)})} L_1
 \le L_1 \rho^{-  \bar \alpha }  e^{\mu\bar \alpha a_0} r_{n+1}^{ \bar \alpha }.$$

Note that for all $0\le k\le n$ we have $$d(x_{k}, x'_{k} ) \le \rho e^{a_0\kappa _0^{(n)} + \hat \gamma _0^{(k)}} \le \rho e^{\kappa_0^{( k)}  } \le \delta. $$ 
Also note that, as $\beta\kappa_n < \kappa_n \theta + \beta \hat \gamma_n$ for all $n$, we have $a_0\ge \frac{\theta}{\beta}\ge \frac{\theta}{\beta\bar \theta}$.
From  Claim \ref{claim:kkklk}, for all $0\le k\le n$  we have 
$$(\xi_{k}, \zeta_{k})\le e^{c_0 k} (r_n)^{\bar \theta\beta} = e^{c_0 k}  \rho ^{\bar \theta\beta}  e^{  \kappa_0^{(n)}  a_0 \bar \theta  \beta}\le  \rho ^{\bar\theta \beta } e^{\kappa_0^{(n)}    \theta  }
.$$



From the properties of the maps $\pi_n$ we have 
	$d(x_{n},y^n_{n}) \le C_2 e^{\kappa_0^{(n)}} \rho $,
	$d(x'_{ n},y'^n_{n}) \le C_2 e^{ \kappa _0^{(n)} } \rho $,
	$\DPP (v_{ n}, w^n_ {n}) \le C_2 e^{\kappa _0^{(n)} \bar\theta    } \rho^{\bar\theta }   $, and 
	 $\DPP (v'_{n}, w'^n_{n}) \le C_2 e^{\kappa_0^{(n)} \bar\theta    } \rho^{\bar\theta } .$
Thus 
	 $$d(y_{n}^n,y'^n_{n}) \le (1+2C_2)  e^{\kappa _0^{(n)} }\rho \le   (1+2C_2) ^{\bar \theta\inv} e^{\kappa _0^{(n)} }\rho$$
and 
	$$ \DPP (w_{ n}^n, w'^n_{ n}) \le  (1+2C_2) e^{\kappa _0^{(n)}  \theta    } \rho^{\bar \theta \beta }.$$
From Lemma \ref{lem:1,1} we have  
	$$d(y^n,y'^n) \le (1+2C_2)  e^{\kappa_0^{(n)}  - \gamma_0^{(n)}}\rho \le (1+2C_2)  e^{\mu\bar \alpha a_0} \rho^{1-\bar \alpha} (r_{n+1})^{\bar \alpha  }$$
		and 
(from Lemma \ref{lem:1,1} 
 with $r= (1+2C_2) ^{\bar \theta\inv}   \rho<\delta$)
%
	\begin{align*} 
	\DPP(w^n,w'^n) &\le 
 (1+2C_2)    \rho^{\bar \theta }  e^{\kappa_0^{(n)} \theta   + (1+\alpha) (\hat \gamma _0^{(n)} -  \gamma_0^{(n)})  } 
\le (1+2C_2) e^{\mu\theta a_0 \bar \alpha}  ( r_{n+1} )^{\bar \alpha}.\end{align*}  
%
%
%
It follows that there is some uniform $K_3>0$ so that  
$$d(\zeta, \zeta') \le K_3 r_{n+1}^{\bar \alpha }\le   K_3 (r_{n+1}^{\bar \theta})^{ \bar \alpha }$$ whence 
$d(\zeta, \zeta') \le K_3 d(\xi, \xi')^{\bar \alpha}$.
\end{proof}

\begin{claim}
The function $(x,v)\mapsto \log\|D_xh (v)\| $ is $ (\bar \theta \bar \alpha) $-H\"older on $SW^c_0(p,1)$.  
\end{claim}
\begin{proof}
We retain all notation from the previous proof.  In particular,  $d(x,x') \le  r_n $; $\DPP (v, v') \le (r_n)^{\bar \theta}  $;
	 and either $d(x,x')\ge r_{n+1}$ or $\DPP (v, v') \ge r_{n+1}^{\bar \theta}$.

Recall that for all $(x,v)\in SW^c_0(p)$ we have $\|\Delta_n(x,v)\| - \|D_x h (v) \| \to 0$ as $n\to \infty$.  Moreover,
from \eqref{eq:lllllll}  we have that for some uniform $ K_4$ that 
\begin{equation}
\label{eq:jjul}
\left|\|\Delta_n(x,v)\| - \|D_x h (v) \| \right|   \le  K_4 e^{\bar \theta \kappa_0^{( n)}}     \le K_5 (r_{n+1}) ^{   \bar \alpha} \end{equation}

We have that
\begin{align*}
\log(\|\Delta_n(x,v) \| )
		&=   \sum_{j= 0 } ^{n-1}  \log \|D_{x_{j}}f_{j}  v_{j}  \| - \log \|D_{y_{j}}f_{j} w_{j} \| .
\end{align*}
Let $L_2\ge \|Dh\|$.  Then for all $0\le j\le n$ we have\begin{itemize}

\item $d(x_{j}, x'_{j})\le e^{\hat \gamma_0^{(k)}}d(x,x')$	whence $d(y_j, y'_j)\le L_2e^{\hat \gamma_0^{(k)}}d(x,x')$

\item  $d(v_{j}, v'_{j}) \le  e^{c_0 j} d ((x,v),(x',v')) ^\beta$
\item $d(w_{j}, w'_{j}) \le K_3 d(v_{j}, v'_{j})^{\bar \alpha} \le K_3 e^{\bar \alpha c_0 j} d ((x,v),(x',v')) ^{\bar \alpha\beta}.$

\end{itemize}

%
%
%
Then for some uniform choice of $K_6$, $K_7$  	and $K_8$ we have 
	\begin{align*}
 |\log \| & \Delta_n(x,v)\| -  \log \|\Delta_n(x',v')\|  | \\
		&\le \sum_{j= 0 } ^{n-1}  \left| \log\| D_{x_{j}}f_{j}  v_{j}\| - \log \|D_{y_{j}}f_{j} w_{j} \| - 
		\log \| D_{x'_{j}}f_{j}  v'_{j} \| +\log \|D_{y'_{j}}f_{j} w'_{j} \| \right| \\
		&\le \sum_{j= 0 } ^{n-1}   \left| \log \|D_{x_{j}}f_{j}  v_{j} \| - 
		\log \| D_{x'_{j}}f_{j}  v'_{j} \| \right|  \\ &\quad + \sum_{j= 0 } ^{n-1} \left |  \log\| D_{y_{j}}f_{j} w_{j} \| - \log\|  D_{y'_{j}}f_{j} w'_{j} \| \right| \\
		&\le \sum_{j= 0 } ^{n-1} C_1^2  d(v_{j}, v_{j}') 
			+ C_3  d(x_{j}, x_{j}') ^\beta + C_1^2  d(w_j, w'_j)  + C_3d(y_{j}, y_{j}')^\beta \\
		& \le K_6  n \left(e^{\hat \gamma_0^{(n)}\beta}d(x,x')^\beta +  e^{c_0 n}  d(\xi, \xi') ^\beta + e^{\bar \alpha c_0 n}  d(\xi, \xi') ^{\bar \alpha \beta} \right) \\
		& \le K_7   \left(e^{\epsilon n+ \hat \gamma_0^{(n)}\beta}d(x,x')^\beta +  e^{(c_0+\epsilon)  n}  d(\xi, \xi') ^\beta + e^{\bar \alpha (c_0+\epsilon)  n}  d(\xi, \xi') ^{\bar \alpha \beta} \right) \\
		 		& \le K_7  (e^{\bar\theta \kappa_0^{(n)} }	+e^{\bar \theta \kappa_0^{(n)} }		+e^{\bar \alpha \bar \theta \kappa_0^{(n)} })\\
		& \le K_8 (r_{n+1})^{\bar \theta \bar \alpha } \le K_8 (r_{n+1}^{\bar \theta})^{\bar \theta \bar \alpha }
		\end{align*}
Combined with \eqref{eq:jjul}, it follows that $(x,v)\mapsto  \| D_xh(v)\| $   is $
(\bar \theta \bar \alpha  )$-H\"older   on $SW^c(p,1)$.    
\end{proof}

\section{Ergodicity of partially hyperbolic diffeomorphisms}\label{sec:3}

We refer the reader to \cite{MR2630044} for   definitions and complete arguements.  Let $M$ be a compact manifold and for $\beta>0$  let $f\colon M\to M$ be a $C^{1+\beta}$ diffeomorphism.  We assume $f$ admits a partially hyperbolic splitting $E^s(x)\oplus E^c(x) \oplus E^u(x)$; that is there are functions  $$ \mu (x)<\nu(x)<\gamma(x) < \hat \gamma(x) \inv <\hat \nu(x) \inv <\hat \mu(x) \inv $$   with $\hat \nu(x), \nu(x) <1$, and  $\hat \gamma (x)\le \gamma(x)$ such that 
\begin{itemize}
	\item $  \mu(x)   \|v\| \le \|D_x f v\| \le   \nu (x)  $ for all $v\in E^s(x)$
		\item $  \gamma (x)   \|v\| \le \|D_x f v\| \le \hat \gamma (x )\inv  $ for all $v\in E^c(x)$
			\item $ \hat \nu (x)\inv    \|v\| \le \|D_x f v\| \le  \hat  \mu (x)\inv  $ for all $v\in E^u(x)$.
\end{itemize}

\begin{theorem}\label{thm:BW}
Let $f\colon M\to M$ be a volume preserving, essentially accessible, partially hyperbolic $C^{1+\beta} $ diffeomorphism. Let $\bar \theta<\beta$ be such that 
$$\nu(x) \gamma\inv (x) < \mu(x) ^{\bar \theta}, \quad \hat \nu(x) \hat \gamma\inv (x) < \hat \mu(x) ^{\bar \theta}.$$
 Assume $f$ satisfies the strong center bunching hypothesis: for some $0<\theta <\bar \theta$
 with \begin{equation}\label{eq:superbunch}\nu(x) ^\theta \ge \nu(x)^\beta \gamma(x) ^{-\beta},\quad  \hat \nu(x) ^\theta \ge \hat  \nu(x)^\beta \hat \gamma(x) ^{-\beta}\end{equation}
 we have 
$$\max\{\nu(x), \hat \nu(x)\} ^{ \theta} \le \gamma \hat \gamma\inv.$$
Then $f$ is ergodic and has the $K$-property.  
\end{theorem}

Noting that $\theta= \bar \theta \beta$ satisfies \eqref{eq:superbunch}, we have as the following corollary.  

\begin{corollary}
If in Theorem \ref{thm:BW}  we have 
$$\max\{\nu(x), \hat \nu(x)\} ^{ \bar \theta\beta} \le \gamma \hat \gamma\inv$$
then $f$ is ergodic and has the $K$-property.  
\end{corollary}
\begin{remark}
Theorem \ref{thm:difflholonomiesspecial} establishes that the smoothness of unstable holonomies inside center-unstable manifolds for a choice of the globalized dynamics.  In the language of \cite{MR2630044}, this establishes the smoothness of holonomy maps by \emph{fake stable manifolds} inside of \emph{fake center-stable manifolds}.  

In the case that the partially hyperbolic diffeomorphism $f\colon M\to M$ is dynamically coherent, one could adapt the proof of Theorem \ref{thm:difflholonomiesspecial} to show the that the holonomy maps by true stable manifolds inside  true center-stable manifolds is $C^{1+\text{H\"older}}$.  
\end{remark}

\section{Ledrappier--Young entropy formula}\label{sec:4}

For  $\beta>0$, let  $f\colon M\to M$ be a $C^{1+\beta}$ diffeomorphism of a compact $k$-dimensional manifold $M$.  Let $\mu$ be an ergodic, $f$-invariant Borel probability measure.
We have the following generalizations of the main results of \cite{MR819556, MR819557}.  
\begin{theorem}[{\cite[Theorem A] {MR819557}}]

 $h_\mu(f)$ satisfies the Pesin entropy formula if and only if $\mu$ has the SRB property.  
\end{theorem}
\begin{theorem}[{\cite[Theorem C'] {MR819557}}]
For any measure $\mu$, the entropy formula of \cite[Theorem C'] {MR819557} remains valid for $h_\mu(f)$.
\end{theorem}

\subsection{Lyapunov charts}
Given  $M$, $f\colon M\to M$ and $\mu$ as above, let  $\Lambda$ denote the set of \emph{regular points}, $T_x M = \oplus E^i(x)$ the Oseledec's splitting, and $\lambda_1>\lambda_2 \dots >\lambda_p$ denote the Lyapunov exponents.  
  Fix a decomposition $\R^k = \oplus \R^i$ where $\dim \R^i = m_i$ is the dimension of $E_i(x)$ for  $x\in \Lambda$.  Define the norm $\|\cdot\|'$ on $\R^d$ as follows: writing $v= \sum v_i$, set  $\|v\|'= \max\{ \|v_i\|\}$ where $\|v_i\|$ restricts to the standard norm on each $\R^i$.
Let $\lambda_0= \max\{|\lambda_1|, | \lambda_p|\}$.  

\def\invco{^{(-1)}}
Fix a background Riemannian metric  and induced distance on $M$.  We have the following standard construction.  (See for example \cite[Appendix]{MR819556}, \cite[\S 2]{MR730270}.)
\begin{proposition}\label{lplplplp} 
For every sufficiently small $0<\hat \epsilon<1$ there is a 
measurable function $\hat \ell\colon \Lambda\to [1,\infty)$ and a measurable family of  $C^\infty$ embeddings  $\{\hat \Phi_x, x\in \Lambda\}$  with the following properties: 
\begin{enumlemma}
	\item $\hat \Phi_x\colon B(0, \hat \ell(x) \inv)\to M$ is a $C^\infty$ diffeomorphism with \label{LC:1} $\hat \Phi_x(0) = x$;
	\item \label{LC:2}$D_0\hat  \Phi_x\R^{i}= E^i(x)$;
	\item \label{LC:31} the map $\hat  f_x \colon B(0,  e^{-\lambda_0 - 3 \hat \epsilon}\hat \ell(x) \inv) \to  B(0, \hat \ell(f(x)) \inv) $ 
	 given by $$\hat  f_x(v) =\hat  \Phi_{f(x) }\inv \circ f\circ\hat  \Phi_x(v)$$
	is  well-defined;
	\item \label{LC:41}$D_0\hat  f_x\R^i= \R^i$, 
	and for $v\in \R^i$, $$ e^{\lambda_i- \hat \epsilon} \|v\|' \le \|D_0\hat  f_x v\|'\le e^{\lambda_i + \hat \epsilon}\|v\|';$$ 
	\item \label{LC:51} $\Hol (D\hat  f_x)\le \hat \epsilon( \hat \ell(x))^{\beta}$ whence $\lip\left (\hat  f_x - D_0\hat  f_x\right) \le \hat \epsilon$; 
\item similar statements to \ref{LC:31},\ref{LC:41} and \ref{LC:51} hold for $f\inv$;
	\item \label{LC:7}there is a uniform $k_0$ so that 
	$\hat \ell(x)\inv \le \lip(\hat  \Phi_x)\le k_0$;%
	\item $e^{-\hat \epsilon} \le \dfrac{\hat \ell(f(x))}{\hat \ell(x)} \le e^{\hat \epsilon}$.
\end{enumlemma}

\end{proposition}

Given $0< \hat \epsilon<1$ and corresponding function $\hat \ell\colon \Lambda\to [1, \infty) $ as in Proposition \ref{lplplplp}
define new charts $  \Phi_x \colon B(0, 1) \to M$ by rescaling: $$  \Phi_x(v) := \hat  \Phi_x(\hat  \ell(x) \inv v).$$
We check with $\epsilon = 4\hat \epsilon$ and $\ell(x) = (\hat \ell(x))^2$ 
  that for every $x\in \Lambda$
\begin{enum2}
	\item \label{LC:1} $  \Phi_x\colon B(0, 1)\to M$ is a $C^\infty$ diffeomorphism with \label{LC:1} $  \Phi_x(0) = x$;
	\item \label{LC:2}$D_0   \Phi_x\R^{i}= E^i(x)$;
	\item \label{LC:3}the map $\td f_x \colon B(0, e^{-\lambda_0 - 2  \epsilon} ) \to B(0, 1) $ 
	 given by $$\td f_x(v) :=   \Phi_{f(x) }\inv \circ f\circ   \Phi_x(v) = 
	\hat   \ell(f(x))  \left(\hat  f_x \left(\hat \ell(x) \inv v\right)\right) $$
	is  well-defined;
	\item \label{LC:4} $D_0\hat  f_x\R^i= \R^i$, 
	and for $v\in \R^i$, $$ e^{\lambda_i-   \epsilon} \|v\|' \le \|D_0\td  f_x v\|'\le e^{\lambda_i +    \epsilon}\|v\|';$$ 
	\item \label{LC:5} $\Hol (D\td f_x)\le   \epsilon$ whence $\lip(\td f_x - D_0\td f_x) \le    \epsilon$;
\item similar statements to \ref{LC:3},\ref{LC:4} and \ref{LC:5} hold for $f\inv$;
		\item \label{LC:7}there is a uniform $k_0$ so that 
	$\ell(x)\inv  \le \lip(\td  \Phi_x)\le k_0;$%
	\item\label{LC:8}  $e^{-  \epsilon} \le \dfrac{\ell(f(x))}{\ell(x)}\le  e^{    \epsilon}$.
	
\end{enum2}
Indeed, \ref{LC:1}, \ref{LC:2}, \ref{LC:3}, \ref{LC:7}, and \ref{LC:8} follow from construction.
For \ref{LC:4},  and \ref{LC:5}, note that for $u\in B(0,1)$, and $\xi\in \R^k$ with $\|\xi\|'= 1$, 
	 $$D_u\td  f_x (\xi) = \hat \ell(f(x)) D_{\hat \ell(x) \inv u }\hat f_x (\hat \ell (x) \inv\xi  )= \dfrac{\hat \ell(f(x))}{\hat \ell(x)}D_{\hat \ell(x) \inv u }\hat  f_x(\xi)  $$
hence 
$$D_0\td  f_x (\xi) = \dfrac{\hat \ell(f(x))}{\hat \ell(x)} D_0 \hat  f_x(\xi)$$
and 
\begin{align*}
\|D_u\td f_x(\xi) - D_v\td f_x(\xi)\|'
&	 = \dfrac{\hat \ell(f(x))}{\hat \ell(x)}\|D_{\hat \ell(x) \inv u }\hat f_x (\xi  )- 
D_{\hat \ell(x) \inv v }\hat f_x (\xi  )\|'\\
&	 \le  \dfrac{\hat \ell(f(x))}{\hat \ell(x)}\Hol(D\hat f_x ) {\|\hat \ell(x) \inv u - \hat \ell(x) \inv v\|'}^\beta \\
&	 \le  \dfrac{\hat \ell(f(x))}{\hat \ell(x)}\hat \epsilon \hat \ell(x) ^\beta \left(\hat \ell(x) \inv \right)^\beta {\| u - v\|'}^\beta \\
&	 =  \dfrac{\hat \ell(f(x))}{\hat \ell(x)}\hat \epsilon {\| u - v\|'}^\beta.  
\end{align*}


The family of maps $\{\Phi_x, x\in \Lambda\}$ is called a family of $(\epsilon,\ell)$ charts.  
Fix   a suitable bump function on $\R^k$ supported inside the unit ball in the $\|\cdot \|'$-norm.    Then there is  a constant $C_1$  such that for all sufficiently small $\epsilon>0$, for any family of  $(\epsilon,\ell)$-charts  $\Phi_x$ as above, with $\td f_x$ as in \ref{LC:3} we may 
extend the maps $\td f_x$ to  $F_x\colon \R^k\to \R^k$ so that  for every $x\in \Lambda$ 
\begin{enumerate}
\item $F_x (u)= \td f_x (u)$ for all $u$ with $\|u\|\le \frac 1 2$; 
\item  $\Lip _{\|\cdot\|'}(F_x- D_0\td f_x)<\epsilon$;
\item  $\Lip _{\|\cdot\|'}(F_x\inv - (D_0\td f_x)\inv)<\epsilon$;
\item $\hol_{\|\cdot\|'} (DF_x ) \le C_1$;  
\item $\hol _{\|\cdot\|'}(DF_x ) \le C_1$.  
\end{enumerate}
Furthermore, taking $\epsilon>0$ sufficiently small we can ensure all relevant estimates remain true in the  Euclidean norm $\|\cdot\|$.

Given sufficiently small $\epsilon>0$,  fix a family of $(\epsilon,\ell)$-charts $\{ \Phi_x : x\in \Lambda\}$ as above.  
Let $0<\delta<1$ be a reduction factor. 
For $x\in \Lambda$, let 
$$S^{cu}_{\delta,x}:= \{ y\in \R^k: \|  \Phi_{f^{-n}(x)}\inv \circ f^{-n} \circ   \Phi_x (y)\|<\delta \text{ for all $n\ge 0$}\}.$$
For $x\in \Lambda$,  $\star\in \{s,u,c,su,cu\}$, and $v\in \R^k$ let $\td W^\star_x(v)$ be the corresponding ``fake'' manifold through the  point $v$ constructed in Proposition \ref{prop:HPS} using the sequence of globalizations $F_x$ along the orbit $f^j(x)$. 

From the dynamics inside charts we obtain the following.
\begin{claim} For all sufficiently small $\delta>0$ we have   
 $S^{cu}_{\delta,x} \subset W^{cu}_x (0).$
\end{claim}
From Corollary \ref{thm:lipholonomies}, it follows (for sufficiently small $\delta>0$)  that the Lipschitz property of holonomies along unstable manifolds inside the center-unstable sets $S^{cu}_{\delta,x}$ 
derived in  \cite[(4.2)]{MR819556} for $C^2$ maps holds for $C^{1+\beta}$ maps.    We similarly obtain that the holonomies along ``fake'' $W^i$ manifolds is Lipschitz inside $W^{i+1}$ manifolds.  This replaces   the Lipschitz estimates \cite[Lemma 8.2.5, (8.4)]{MR819557}.  
In particular, \cite[Proposition 5.1]{MR819556} and  \cite[Proposition 11.2]{MR819557}  remain valid  for $C^{1+\beta}$ diffeomorphisms.  It follows that the results of \cite{MR819556} and \cite{MR819557} hold for $C^{1+\beta}$ diffeomorphisms.

\end{document}